\documentclass{amsart}

\usepackage{cite}
\usepackage{hyperref}
\usepackage[shortlabels]{enumitem}
\usepackage{mathtools}

\newtheorem{theorem}{Theorem}[section]

\theoremstyle{definition}
\newtheorem{definition}[theorem]{Definition}
\theoremstyle{remark}
\newtheorem{remark}[theorem]{Remark}
\newtheorem{corollary}[theorem]{Corollary}

\numberwithin{equation}{section}
\allowdisplaybreaks

\newcommand{\ex}{\mathbb E}
\newcommand{\RR}{\mathbb R}
\newcommand{\prob}{\mathbb P}

\calclayout

\begin{document}
\title[Quasihelix properties of selected Volterra Gaussian processes]{Quasihelix properties\\ of selected Volterra Gaussian processes}
\author{Yuliya Mishura}
\address[Y. Mishura]{Department of Probability, Statistics and Actuarial Mathematics, Taras Shevchenko National University of Kyiv,
Volodymyrska 64, 01601 Kyiv, Ukraine}
\email{yuliyamishura@knu.ua}
\thanks{Y.M. is partially supported by the Japan Science and Technology Agency CREST, project reference number JPMJCR2115}

\author{Kostiantyn Ralchenko}
\address[K. Ralchenko]{Department of Probability, Statistics and Actuarial Mathematics, Taras Shevchenko National University of Kyiv,
Volodymyrska 64, 01601 Kyiv, Ukraine}
\email{kostiantynralchenko@knu.ua}
\address[K. Ralchenko]{School of Technology and Innovations, University of Vaasa, P.O. Box 700, Vaasa, FIN-65101, Finland}
\thanks{K.R. is supported by the Research Council of Finland, decision number 367468.}




\begin{abstract}
We study local quasihelix and generalized quasihelix properties of several Gaussian Volterra processes with tempered, power-weighted, and logarithmic kernels, including tempered fractional Brownian motions and generalized fractional Brownian motion-type processes. These properties depend significantly on the values  of the parameters involved, and we consider all possible cases in detail.
\end{abstract}
\maketitle

\section{Introduction}

Gaussian Volterra processes form a flexible class of Gaussian processes whose covariance structure is generated by deterministic Volterra-type kernels. Although their finite-dimensional distributions remain Gaussian, suitable choices of the kernel allow one to model non-Markovian and non-semimartingale dynamics, memory effects, nonstationarity, and different forms of local regularity. Except in special cases, such as fractional Brownian motion, these processes typically do not have stationary increments, which makes their path properties more delicate and strongly dependent on the kernel.

Such processes arise in a wide range of applications, including anomalous diffusion in complex physical systems \cite{anomdiff1,anomdiff2}, turbulence and geophysical flows \cite{turbulence}, biological and medical applications such as intracellular transport and brain structure modeling \cite{brain,wang2022anomalous}, energy and electricity markets \cite{electricity}, and modern mathematical finance, in particular, rough volatility modeling \cite{volatility1,volatility2,volatility3}. Fractional Brownian motion is the classic example, but recent work has considered various modifications of the fractional and Riemann--Liouville kernels by means of tempering, weighting, or logarithmic factors. These include exponentially tempered fractional Brownian motion \cite{Meerschaert2013}, power-weighted or power-tempered variants \cite{Azmoodeh2021}, and processes based on logarithmic Hadamard-type kernels \cite{Beghin2025,Beghin2026}.

The purpose of this paper is to study quasihelix and generalized quasihelix properties for several such Gaussian Volterra processes. The notion of a quasihelix was introduced by Kahane \cite{Kahane1985}; generalized quasihelices and their use in deriving upper and lower bounds for expected maxima were developed in \cite{Borovkov2017}. Since these properties are governed by the behavior of the incremental variance, they provide a useful tool for analyzing pathwise behavior and for obtaining estimates relevant to the study of maxima and numerical approximations.

We consider Gaussian Volterra processes whose integral representations start either from zero or from $-\infty$. For processes started from zero, we analyze the kernels of three types: an exponentially tempered Riemann--Liouville kernel, a power-weighted Riemann--Liouville kernel, and a logarithmic Hadamard-type kernel. In the power-weighted case, we distinguish between intervals separated from the origin and intervals of the form $[0,T]$, where the singularity at zero leads to weaker generalized quasihelix bounds. We also derive exact local asymptotics, as the length of the incremental interval tends to zero,  of the incremental variance on intervals $[T_1,T_2]$, $0<T_1<T_2$.

For processes started from $-\infty$, we collect known quasihelix properties of tempered fractional Brownian motion and tempered fractional Brownian motion of the second kind and explain their extension to arbitrary compact intervals. We also consider a power-weighted process introduced in \cite{WangXiao2022}, decompose it into a zero-started component and an independent component supported on $(-\infty,0]$, and establish the corresponding local quasihelix properties.

The paper is organized as follows. Section~\ref{sec:prelim} introduces the two classes of Gaussian Volterra processes under consideration  and recalls the definitions of quasihelix and generalized quasihelix properties. Section~\ref{sec:quasihelix} contains the main results, describing in detail all cases  of possible values of parameters and respective quasihelix properties. Subsection~\ref{sec:gvp-zero} treats processes started from zero, while Subsection~\ref{sec:gvp-inf} is devoted to processes started from $-\infty$.

\section{Preliminaries}
\label{sec:prelim}
We begin with basic definitions and introduce the main concepts used throughout the paper.

\subsection{Two classes of Gaussian Volterra processes}

Let $(\Omega,\mathcal{F},\prob)$ be a complete probability space supporting a two-sided Brownian motion
$W=\{W_t,\ t\in\RR\}$. We consider two classes of centered Gaussian processes of Volterra type. Both are defined by stochastic integral representations, but they differ in the lower limit of integration.

The first class consists of processes started from zero:
\begin{equation*} 
    X_t = \int_0^t K(t,s)\,dW_s, \qquad t\ge0,
\end{equation*}
where $K\colon [0,\infty)^2\to\RR$ is a deterministic measurable kernel. To ensure that the process is well defined and has finite variance, we assume that
\begin{equation*}
    \int_0^t K(t,s)^2\,ds < \infty, \qquad t\ge 0.
\end{equation*}
The covariance function of $X$ is then given by
\[
R(t,s)=\ex[X_tX_s]
=
\int_0^{t\wedge s} K(t,u)K(s,u)\,du.
\]

The second class consists of processes started from $-\infty$:
\begin{equation*} 
    Y_t = \int_{-\infty}^t K(t,s)\,dW_s, \qquad t\ge0,
\end{equation*}
where $K\colon\RR^2\to\RR$ is a deterministic measurable kernel. In this case, we assume that
\begin{equation*}
    \int_{-\infty}^t K(t,s)^2\,ds < \infty, \qquad t>0,
\end{equation*}
and the covariance function of $Y$ is
\[
R(t,s)=\ex[Y_tY_s]
=
\int_{-\infty}^{t\wedge s} K(t,u)K(s,u)\,du.
\]

Throughout the paper, we use the standard $L_2(\prob)$-norm
defined as $\|\cdot\|_2= (\ex|\cdot|^2)^{1/2}$.

\subsection{Quasihelices and generalized quasihelices}

According to the seminal results of Talagrand \cite{Talagrand2005,Talagrand2014}, bounds for the expected maximum of a continuous Gaussian process are closely related to the behavior of its incremental variance. More precisely, such bounds are connected with whether the Gaussian process has the quasihelix or generalized quasihelix property. The notion of a quasihelix was introduced by Kahane \cite{Kahane1985} and later extended to the notion of a generalized quasihelix in \cite{Borovkov2017}. Following this terminology, we use the following definitions.

\begin{definition}
A Gaussian process $G=\{G_t,\ t\ge0\}$ is called
\begin{itemize}
    \item[$(i)$] a \emph{$\rho$-quasihelix} if there exist $\rho>0$ and constants $C_1,C_2>0$ such that, for all $t,s\in[0,\infty)$,
    \begin{equation*}
        C_1|t-s|^{\rho}
        \le
        \|G_t-G_s\|_2
        \le
        C_2|t-s|^{\rho};
    \end{equation*}

    \item[$(ii)$] a \emph{local $(\rho,T_1,T_2)$-quasihelix} on the interval $[T_1,T_2]$ if there exist $\rho>0$ and constants $C_1(T_1,T_2),C_2(T_1,T_2)>0$ such that, for all $t,s\in[T_1,T_2]$,
    \begin{equation}\label{loc_quasihelix}
        C_1(T_1,T_2)|t-s|^{\rho}
        \le
        \|G_t-G_s\|_2
        \le
        C_2(T_1,T_2)|t-s|^{\rho};
    \end{equation}

    \item[$(iii)$] a \emph{local $(\rho_1,\rho_2,T_1,T_2)$-generalized quasihelix} on the interval $[T_1,T_2]$ if there exist $\rho_1,\rho_2>0$ and constants $C_1(T_1,T_2),C_2(T_1,T_2)>0$ such that, for all $t,s\in[T_1,T_2]$,
    \begin{equation}\label{loc_gen_quasihelix}
        C_1(T_1,T_2)|t-s|^{\rho_1}
        \le
        \|G_t-G_s\|_2
        \le
        C_2(T_1,T_2)|t-s|^{\rho_2}.
    \end{equation}
\end{itemize}
\end{definition}

\begin{remark}
If \eqref{loc_gen_quasihelix} holds on $[T_1,T_2]$, then necessarily $\rho_1\ge \rho_2$. Indeed, otherwise the lower bound would eventually dominate the upper bound as $t\to s$.

Clearly, any local quasihelix or generalized quasihelix property on an interval $[T_1,T_2]$ is inherited by every subinterval $[T_1',T_2']\subseteq[T_1,T_2]$. However, as will be seen below, the validity of such a property on every interval $[T_1,T_2]$ with $0<T_1<T_2$ does not necessarily imply its validity on intervals of the form $[0,T]$. This distinction is caused by possible singularities of the kernels of Gaussian Volterra processes at the origin.
\end{remark}
\section{Gaussian Volterra processes and their quasihelix properties}
\label{sec:quasihelix}

Let us consider quasihelix properties of several classes of Gaussian Volterra processes. Some of them were originally introduced under the names of tempered fractional Brownian motion and generalized fractional Brownian motion; however, we shall avoid this terminology here to prevent possible confusion.

\subsection{Gaussian Volterra processes starting from zero}
\label{sec:gvp-zero}

We consider three specific Gaussian Volterra processes defined by distinct kernels $K(t,s)$ and started from zero. We analyze their properties in the context of conditions \eqref{loc_quasihelix} and \eqref{loc_gen_quasihelix}.

First, the process introduced in \cite{Mishura2024} is defined via the exponential-tempered Riemann--Liouville kernel. For parameters $\lambda > 0$ and $\alpha > -\frac12$, it is given by
\begin{equation}\label{tfbm}
    U^{(1)}(t) = \int_0^t e^{-\lambda u}(t - u)^{\alpha} \, dW_u.
\end{equation}

Next, the process introduced in \cite[Remark 5.1]{Pang2019} and further studied in \cite{Ichiba2022, WangXiao2022, Wang2024} is defined via the power-tempered Riemann--Liouville kernel. For parameters $\gamma \in (0,1)$ and $\alpha > -\frac12$, it is given by
\begin{equation}\label{gfbm}
    U^{(2)}(t) = \int_0^t u^{-\gamma/2} (t - u)^{\alpha} \, dW_u.
\end{equation}

Finally, the process introduced in \cite{Beghin2025} is constructed using the Hadamard fractional integral kernel. For a parameter $\alpha > 0$, it is given by
\begin{equation}\label{hfbm}
    U^{(3)}(t) = \int_0^t \left( \log \frac{t}{u} \right)^{\frac{\alpha - 1}{2}} \, dW_u.
\end{equation}
\begin{remark} 
All three processes are centered Gaussian processes. Their second moments are given by
\begin{gather*}
\ex\bigl(U^{(1)}(t)\bigr)^2
= e^{-2\lambda t} \int_0^t e^{2\lambda u} u^{2\alpha} \, du,\;
 \ex\bigl(U^{(2)}(t)\bigr)^2
=   B\bigl(1-\gamma,\, 1+2\alpha\bigr)\, t^{-\gamma+2\alpha+1},
\\
\ex\bigl(U^{(3)}(t)\bigr)^2
=   \Gamma(\alpha)\, t.
\end{gather*}
In this connection, we see that it is convenient to normalize
$U^{(3)}$ by $(\Gamma(\alpha))^{-1/2}$ and get just $t$
(this normalizing factor is introduced in~\cite{Beghin2025}). Also, normalization by a constant multiplier is  possible for $U^{(2)}$, but not for  $U^{(1)}$.
Since we are more interested in the trajectory-wise properties of $U^{(i)}$
as functions of $t$, and not in the specific values of normalizing multipliers, we shall not apply normalization.
\end{remark}
 
The next theorems collect the quasihelix and generalized quasihelix properties of the processes \eqref{tfbm}--\eqref{hfbm}. Some of these properties have already been established, so we only provide relevant references and prove only those properties that have not been established previously.

\begin{theorem}\label{theor1-1a}
\leavevmode
\begin{enumerate}[1),left=5pt]
\item \label{th1-1)}
 \begin{enumerate}[(i)]
    \item For any fixed $\lambda > 0$ and $\alpha\in(-1/2, 1/2)$, the process $U^{(1)}$ is a local $(\alpha+\frac12, 0, T)$-quasihelix for any $T>0$.
    \item Let $\alpha = 1/2$. Then, for any fixed $\lambda > 0 $ and any $\varepsilon\in(0,1)$, the process $U^{(1)}$ is a  local $(1, 1 - \varepsilon, 0, T)$-generalized quasihelix for any $T>0$.  
    \item Let $\alpha > 1/2$. Then, for any fixed $\lambda > 0$, the process $U^{(1)}$ is a local $(\alpha+\frac12, 1, 0, T)$-generalized quasihelix  for any $T>0$ and $(1, T_1, T_2)$-quasihelix  for any  $0 < T_1 < T_2$.   
\end{enumerate} 

\item \label{th1-2)}
For any $T > 0$, the   process $U^{(3)}$ is local $(\frac12, \frac\alpha2, 0, T)$-generalized quasihelix in the case $\alpha \in (0, 1)$ and local $(\frac\alpha2, \frac12, 0, T)$-generalized quasihelix in the case $\alpha \in (1, 2)$. When $\alpha=1$, the process becomes a standard Wiener process, and thus, it is obviously a local $(\frac12, 0, T)$-quasihelix for all $T > 0$. For $\alpha \ge 2$, the process is a local $(\frac{\alpha}{2}, \frac12, 0, T)$-generalized quasihelix for any $T>0$.
\end{enumerate}
\end{theorem}

Statements \ref{th1-1)} and \ref{th1-2)} were proved in \cite{Mishura2024} and \cite{Beghin2025}, respectively. Therefore, the proof is omitted. 

\medskip
The next theorem collects the local quasihelix properties of the process \(U^{(2)}\) on intervals separated from the origin. Its relation to the corresponding bounds of Wang and Xiao~\cite{WangXiao2022} is explained in Remark~\ref{rem:U2-Wang-Xiao} below.

\begin{theorem}\label{theor1-1b}
Let $\gamma\in[0,1)$. Then, for any $0<T_1<T_2$, the following statements hold.
\begin{enumerate}
\item[$(i)$]
Let $\alpha \in (-1/2, 1/2)$.
Then 
\begin{equation}\label{eq:asymp}
\ex\left(U^{(2)}(t)-U^{(2)}(s)\right)^2
\sim C(\alpha)(t-s)^{2\alpha+1}s^{-\gamma},
\quad \text{as } t\downarrow s,
\end{equation}
for any $s\in[T_1,T_2]$,
where
\begin{equation}\label{eq:Calpha}
C(\alpha) = \int_0^\infty \bigl[(1+z)^\alpha-z^\alpha\bigr]^2\,dz
+ \frac{1}{2\alpha+1}
= \frac{\Gamma(\alpha+1)^2} {\Gamma(2\alpha+2)\cos(\pi\alpha)}.
\end{equation}
Consequently, the process $U^{(2)}$ is a local $(\alpha+\frac12,T_1,T_2)$-quasihelix.

\item[$(ii)$]
If $\alpha=1/2$, then the process  $U^{(2)}$ is a local $(1,1-\varepsilon,T_1,T_2)$-generalized quasihelix for any $\varepsilon\in(0,1)$.

\item[$(iii)$]
If $\alpha > 1/2$, then the process $U^{(2)}$ is a local $  (1,T_1,T_2)$-quasihelix.
\end{enumerate}
\end{theorem}

\begin{remark}\label{rem:U2-Wang-Xiao}
The process \(U^{(2)}\) coincides, up to notation, with the generalized
Riemann--Liouville fractional Brownian motion considered in
\cite{WangXiao2022}. Therefore, some of the bounds in
Theorem~\ref{theor1-1b} overlap with bounds established in that paper.
More precisely, the asymptotic relation in statement~\((i)\) follows from
\cite[Lemma~3.2]{WangXiao2022}, while statement~\((ii)\) follows from the
logarithmic two-sided estimate in \cite[Remark~3.1(i)]{WangXiao2022}.
Statement~\((iii)\) follows from \cite[Lemma~3.1]{WangXiao2022} in the
range \(1/2<\alpha<3/2\). In the proof below we show that the same local
quasihelix conclusion remains valid for all \(\alpha>1/2\).

We include Theorem~\ref{theor1-1b} in order to formulate these estimates in
the language of local quasihelices and to compare them with the other
Volterra processes considered in the paper. For completeness, a direct
derivation of the asymptotic relation in statement~\((i)\), adapted to the
notation and constants used here, is given in Appendix.
\end{remark}

\begin{remark}
The constant $C(\alpha)$ appearing in statement~$(i)$ is the well-known normalizing constant from the Mandelbrot--Van Ness representation of fractional Brownian motion; see, e.g., \cite[Theorem~1.3.1]{Mishura}.
\end{remark}

In what follows, \(C(\alpha,\gamma)\), \(C(T,\alpha,\gamma)\), and
\(C(T_1,T_2,\alpha,\gamma)\) denote positive constants depending only on
the indicated parameters. Their values may change from line to line. 
 
\begin{proof}[Proof of Theorem \ref{theor1-1b}]
Consider the increment $U^{(2)}(t)-U^{(2)}(s)$, where
$0<T_1\le s<t\le T_2$.

$(i)$
For $\alpha\in(-1/2,1/2)$, the asymptotic relation \eqref{eq:asymp} follows from \cite[Lemma~3.2]{WangXiao2022}; see also Appendix for a direct derivation in the present notation. The local quasihelix property follows immediately, because $s^{-\gamma}$ is bounded above and below by positive constants on every interval $[T_1,T_2]$, $0<T_1<T_2$.

$(ii)$ In the borderline case $\alpha = 1/2$ the statement follows from \cite[Remark~3.1(i)]{WangXiao2022}, where the following two-sided bounds were established 
\[
C_1 |t-s|^2 (1 + \bigl|\log|t-s|\bigr|)
\le \ex\left(U^{(2)}(t)-U^{(2)}(s)\right)^2
\le C_2 |t-s|^2 (1 + \bigl|\log|t-s|\bigr|).
\]
Indeed, the lower bound implies a bound of order $|t-s|^2$, while the upper bound satisfies
\[
|t-s|^2(1+|\log|t-s||)\le C_\varepsilon |t-s|^{2-2\varepsilon}
\]
for every $\varepsilon\in(0,1)$.

$(iii)$
We have
\begin{align}\label{jiji1and2}
\ex\big(U^{(2)}(t)-U^{(2)}(s)\big)^2
&= \int_{0}^{s} u^{-\gamma}
\left[(t-u)^{\alpha}-(s-u)^{\alpha}\right]^2 du + \int_{s}^{t} u^{-\gamma}(t-u)^{2\alpha}du \\
&=: J^1 + J^2 .\notag
\end{align}

By the Lagrange theorem, we can transform the term $J^1$   as follows: 
\begin{align*}
  J^1&= \int_{0}^{s} (s-u)^{-\gamma}
\left[(t-s+u)^{\alpha}-u^{\alpha}\right]^2 du=\alpha^2 (t-s)^2\int_{0}^{s} (s-u)^{-\gamma} 
 (\theta_{s.t}+u)^{2\alpha-2}du,  
\end{align*}
where $\theta_{s.t}\in  (0,t-s)$.
Assuming that $1/2< \alpha<1$, we note that   $\theta_{s.t}+u > u$   and consequently, get the following upper bound: 
\begin{align*} 
J^1&\le  \alpha^2 (t-s)^2\int_{0}^{s} (s-u)^{-\gamma}u^{2\alpha-2}du
\le \alpha^2  (t-s)^2s^{2\alpha-\gamma-1}B(1-\gamma, 2\alpha-1). 
\end{align*}
If $2\alpha-\gamma-1<0,$ then $s^{2\alpha-\gamma-1}\le T_1^{2\alpha-\gamma-1}$, while, if $2\alpha-\gamma-1\ge 0,$ then $s^{2\alpha-\gamma-1}\le T_2^{2\alpha-\gamma-1}$.
In any case, if $1/2< \alpha<1$, then there exists a constant $C(\alpha, \gamma, T_1, T_2)$ such that
\[
 J^1 \le  C(\alpha, \gamma, T_1, T_2)(t-s)^2.  
\]
Furthermore, if $\alpha> 1$, we note that  $\theta_{s.t}+u\le T_2$, and the  upper bound takes  a form 
\begin{align*}
  J^1&\le  \alpha^2 (t-s)^2 T_2^{2\alpha -2}\int_{0}^{s} (s-u)^{-\gamma}du\le  \alpha^2(1-\gamma)^{-1}T_2^{1 -\gamma}T_2^{2\alpha -2}(t-s)^2.
  \end{align*}
  Summarizing, for any  $\gamma\in[0,1)$,  $0<T_1<T_2$ and any $\alpha>1/2$, there exists a constant $C(T_1,T_2,\alpha,\gamma)$ such that 
\begin{align}\label{Jiji11-1}
 J^1&\le  C(T_1,T_2,\alpha,  \gamma)(t-s)^2.
\end{align}
For the lower bound, take $u\in[0,T_1/2]$. Then $s-u,t-u\in[T_1/2,T_2]$, and the mean value theorem gives
\[
\bigl|(t-u)^\alpha-(s-u)^\alpha\bigr|
\ge c(T_1,T_2,\alpha)(t-s).
\]
Therefore,
\begin{equation}\label{Jiji11-2}
J^1 \ge c(T_1,T_2,\alpha)(t-s)^2
\int_0^{T_1/2}u^{-\gamma}\,du
\ge c(T_1,T_2,\alpha,\gamma)(t-s)^2.
\end{equation}
  
For any $\alpha>-1/2$, the term  $J^2$ can be  bounded as follows: 
\begin{equation}\label{Jiji2} 
\frac{T_2^{-\gamma}}{2\alpha+1} (t-s)^{2\alpha+1} \le J^2 \le \frac{T_1^{-\gamma}}{2\alpha+1} (t-s)^{2\alpha+1}. 
\end{equation}
Note that for $a<b$ and $s,t\in[T_1,T_2]$ it holds that $(t-s)^{b}\le (T_1+T_2)^{b-a}(t-s)^{a}$. This means that, having two degrees, we are forced to choose the smaller  degree in  the upper bound, and it is advantageous for us to take the smaller degree in  the lower one. With this at hand, note that for  $\alpha> 1/2$ we have that  $2\alpha+1>2, $ and the proof   follows from comparison of \eqref{Jiji11-1}, \eqref{Jiji11-2} and  \eqref{Jiji2}.   
\end{proof}

The results of Theorem~\ref{theor1-1b} essentially rely on the assumption $T_1>0$. When $T_1=0$, one can still establish generalized quasihelix properties, but this requires more delicate estimates and naturally leads to slightly weaker bounds.
The following result gives us the quasihelix properties of the process $U^{(2)}$ on the interval $[0,T]$.

\begin{theorem}\label{theor1-1c}
Let $\gamma\in[0,1)$. Then the following assertions hold.
\begin{enumerate}[1)]
\item Let $\alpha\in(-1/2+\gamma/2,1/2+\gamma/2)$. Then:
\begin{enumerate}
    \item[$(i)$]
There exists a constant $C(\alpha,\gamma)>0$ such that, for any $0\le s<t$,
\[
\ex\big(U^{(2)}(t)-U^{(2)}(s)\big)^2
\le C(\alpha,\gamma)(t-s)^{2\alpha+1-\gamma}.
\]

\item[$(ii)$]
The process $U^{(2)}$ is a local
$(\alpha+\frac12,\alpha+\frac{1-\gamma}{2},0,T)$-generalized quasihelix for any $T>0$.
\end{enumerate}

\item If $\alpha\in[1/2+\gamma/2,1)$, then the process $U^{(2)}$ is a local
$(\alpha+\frac12,\alpha,0,T)$-generalized quasihelix for any $T>0$.

\item If $\alpha\ge1$, then the process $U^{(2)}$ is a local
$(\alpha+\frac12,1,0,T)$-generalized quasihelix for any $T>0$.
\end{enumerate}
\end{theorem}

\begin{remark}
The estimates in Theorem~\ref{theor1-1c} are consistent with the H\"older-continuity result of \cite[Theorem~3.1]{Ichiba2022} for the generalized fractional Brownian motion, which corresponds, up to normalization, to the process $U^{(6)}$ considered below. However, Theorem~\ref{theor1-1c} concerns the zero-started component $U^{(2)}$, and it does not follow directly from \cite[Theorem~3.1]{Ichiba2022}. Moreover, statement~1)(i) gives a global upper bound for the second moment of increments, valid for all $0\le s<t$. Statement~1)(ii) establishes a generalized quasihelix property on intervals of the form
$[0,T]$, which requires both lower and upper bounds for the incremental variance and
is therefore not a consequence of a H\"older-continuity assertion alone. Finally, parts~2) and~3) treat the additional parameter ranges $\alpha\in[1/2+\gamma/2,1)$ and $\alpha\ge1$, which lie outside the usual parameter range of the generalized fractional Brownian motion considered in \cite{Ichiba2022}.
\end{remark}

\begin{remark}
When $\gamma=0$, the process in item~$1)$ reduces to the
Riemann--Liouville fractional Brownian motion
$U^{(2)}(t)=\int_0^t (t-u)^\alpha\,dW_u$,
which is known to be an $(\alpha+1/2)$-quasihelix. Thus, item~$1)(i)$,
which in this case gives
\[
\ex\big(U^{(2)}(t)-U^{(2)}(s)\big)^2
\le C(\alpha)(t-s)^{2\alpha+1},
\qquad 0\le s<t,
\]
extends this upper bound to the power-tempered Riemann--Liouville case.
\end{remark}

\begin{proof}
We first prove statement~$1)(i)$.

Denote $r=\frac{s}{t-s}$.
The two terms in \eqref{jiji1and2} admit the representations
\begin{equation}\label{Jiji11}
\begin{split}
J^1&= \int_{0}^{s} (s-z)^{-\gamma}
\left[(t-s+z)^{\alpha}-z^{\alpha}\right]^2 dz 
\\
&= (t-s)^{2\alpha-\gamma+1}
\int_{0}^{r} (r-z)^{-\gamma}
\left[(1+z)^{\alpha}-z^{\alpha}\right]^2 dz,
\end{split}
\end{equation}
and
\begin{equation}\label{int1.12}
J^{2} = \int_0^{t-s} (u+s)^{-\gamma}(t-s-u)^{2\alpha}\,du  = (t-s)^{2\alpha-\gamma+1}
\int_0^1 (z+r)^{-\gamma}(1-z)^{2\alpha}\,dz.
\end{equation}

Let \(\alpha\in(-1/2+\gamma/2,0]\). If \(\alpha=0\), then \(J^1=0\).
Assume, therefore, that \(\alpha<0\), and put \(\alpha'=-\alpha>0\).
By \eqref{Jiji11},
\[
J^1=(t-s)^{2\alpha-\gamma+1}J^3,
\]
where
\[
J^3=
\int_0^r (r-z)^{-\gamma}
\bigl[(1+z)^\alpha-z^\alpha\bigr]^2\,dz.
\]
Since
\[
0<z^\alpha-(1+z)^\alpha
=
\frac{(1+z)^{\alpha'}-z^{\alpha'}}
{(1+z)^{\alpha'}z^{\alpha'}},
\]
and \(0<(1+z)^{\alpha'}-z^{\alpha'}<1\), we have
\begin{equation}\label{eq:neg-alpha-bound-1}
\bigl[(1+z)^\alpha-z^\alpha\bigr]^2
\le (1+z)^{2\alpha}z^{2\alpha}.
\end{equation}
Moreover, by the mean value theorem,
\[
(1+z)^{\alpha'}-z^{\alpha'}
=
\alpha'(z+\theta_z)^{\alpha'-1},
\qquad \theta_z\in(0,1),
\]
and therefore
\begin{equation}\label{eq:neg-alpha-bound-2}
\bigl[(1+z)^\alpha-z^\alpha\bigr]^2
\le C(\alpha)(1+z)^{2\alpha}z^{-2}.
\end{equation}

If \(r\le1\), then by \eqref{eq:neg-alpha-bound-1},
\[
J^3
\le
\int_0^r (r-z)^{-\gamma}z^{2\alpha}\,dz
=
r^{2\alpha-\gamma+1}
B(1-\gamma,2\alpha+1)
\le C(\alpha,\gamma),
\]
because \(2\alpha-\gamma+1>0\), or equivalently
\(\alpha>-1/2+\gamma/2\).

Let now \(r>1\). We split the integral defining \(J^3\) into the parts
over \((0,1)\) and \((1,r)\). For the first part, again using
\eqref{eq:neg-alpha-bound-1}, we obtain
\[
\int_0^1 (r-z)^{-\gamma}
\bigl[(1+z)^\alpha-z^\alpha\bigr]^2\,dz
\le
\int_0^1 (1-z)^{-\gamma}z^{2\alpha}\,dz
=
B(1-\gamma,2\alpha+1).
\]
For the second part, \eqref{eq:neg-alpha-bound-2} gives
\[
\int_1^r (r-z)^{-\gamma}
\bigl[(1+z)^\alpha-z^\alpha\bigr]^2\,dz
\le
C(\alpha)\int_1^r (r-z)^{-\gamma}z^{2\alpha-2}\,dz.
\]
If \(1<r\le2\), the last integral is bounded by
\[
\int_1^r (r-z)^{-\gamma}\,dz\le \frac{1}{1-\gamma}.
\]
If \(r>2\), we split it into \((1,r/2)\) and \((r/2,r)\). Then
\[
\int_1^{r/2}(r-z)^{-\gamma}z^{2\alpha-2}\,dz
\le
(r/2)^{-\gamma}\int_1^\infty z^{2\alpha-2}\,dz
\le C(\alpha,\gamma),
\]
and
\[
\int_{r/2}^{r}(r-z)^{-\gamma}z^{2\alpha-2}\,dz
=
r^{2\alpha-1-\gamma}
\int_{1/2}^{1}(1-z)^{-\gamma}z^{2\alpha-2}\,dz
\le C(\alpha,\gamma),
\]
because \(2\alpha-1-\gamma<0\). Thus \(J^3\le C(\alpha,\gamma)\) for all
\(r>0\), and consequently
\begin{equation}\label{eq:J1-upper-negative-alpha}
J^1
\le C(\alpha,\gamma)(t-s)^{2\alpha+1-\gamma},
\qquad 0\le s<t.
\end{equation}

By \eqref{int1.12}, for every \(\alpha>-1/2\),
\begin{equation}\label{eq:J2-upper-global}
J^{2}
\le (t-s)^{2\alpha-\gamma+1}
\int_0^1 z^{-\gamma}(1-z)^{2\alpha}\,dz
=
B(1-\gamma,2\alpha+1)(t-s)^{2\alpha-\gamma+1}.
\end{equation}
Therefore, \eqref{eq:J1-upper-negative-alpha} and \eqref{eq:J2-upper-global} yield
\[
\ex\bigl(U^{(2)}(t)-U^{(2)}(s)\bigr)^2
=J^1+J^2
\le C(\alpha,\gamma)(t-s)^{2\alpha+1-\gamma}
\]
for \(\alpha\in(-1/2+\gamma/2,0]\).

It remains to consider \(0<\alpha<1/2+\gamma/2\). We use the same representation
\[
J^1=(t-s)^{2\alpha-\gamma+1}J^3.
\]
In this case the estimate of \(J^3\) is simpler. Since
\((1+z)^\alpha-z^\alpha\le1\), for \(r\le1\) we have
\[
J^3
\le
\int_0^r (r-z)^{-\gamma}\,dz
=
\frac{r^{1-\gamma}}{1-\gamma}
\le \frac{1}{1-\gamma}.
\]
Hence
\[
J^1\le C(\alpha,\gamma)(t-s)^{2\alpha+1-\gamma}.
\]

Let now \(r>1\). As above, we split \(J^3\) into the parts over \((0,1)\) and
\((1,r)\). The first part is bounded by
\[
\int_0^1 (r-z)^{-\gamma}
\bigl[(1+z)^\alpha-z^\alpha\bigr]^2\,dz
\le
\int_0^1 (1-z)^{-\gamma}
\bigl[(1+z)^\alpha-z^\alpha\bigr]^2\,dz
\le C(\alpha,\gamma).
\]
For the second part, by the mean value theorem, for \(z\ge1\),
\[
(1+z)^\alpha-z^\alpha
=
\alpha(z+\theta_z)^{\alpha-1},
\qquad \theta_z\in(0,1),
\]
and therefore
\begin{align*}
\int_1^r(r-z)^{-\gamma}
\bigl[(1+z)^\alpha-z^\alpha\bigr]^2\,dz
&\le \alpha^2\int_1^r(r-z)^{-\gamma}z^{2\alpha-2}\,dz  \\
&= \alpha^2 r^{2\alpha-\gamma-1}
\int_{1/r}^{1}(1-z)^{-\gamma}z^{2\alpha-2}\,dz .
\end{align*}
It remains to check that
\[
f(r)=r^{2\alpha-\gamma-1}
\int_{1/r}^{1}(1-z)^{-\gamma}z^{2\alpha-2}\,dz
\]
is bounded on \([1,\infty)\). If \(0<\alpha<1/2\), then the integral is
\(O(r^{1-2\alpha})\); if \(\alpha=1/2\), it is \(O(\log r)\); and if
\(\alpha>1/2\), it is bounded. Since
\(\alpha<1/2+\gamma/2\), all three cases imply
\(f(r)\le C(\alpha,\gamma)\). Thus \(J^3\le C(\alpha,\gamma)\), and hence
\[
J^1\le C(\alpha,\gamma)(t-s)^{2\alpha+1-\gamma}.
\]
Combining this with \eqref{eq:J2-upper-global}, we obtain
\[
\ex\bigl(U^{(2)}(t)-U^{(2)}(s)\bigr)^2
\le C(\alpha,\gamma)(t-s)^{2\alpha+1-\gamma}
\]
for \(0<\alpha<1/2+\gamma/2\). This proves statement~$1)(i)$.

To prove statement~$1)(ii)$, it remains to obtain a lower bound. By
\eqref{jiji1and2}, for any $T>0$ and $0\le s<t\le T$,
\begin{equation}\label{eq:J2-lower}
J^2 = \int_s^t u^{-\gamma}(t-u)^{2\alpha}\,du
\ge T^{-\gamma}\int_s^t(t-u)^{2\alpha}\,du
= \frac{T^{-\gamma}}{2\alpha+1}(t-s)^{2\alpha+1}.
\end{equation}
Consequently,
\begin{equation}\label{eq:lower-bound}
\ex\big(U^{(2)}(t)-U^{(2)}(s)\big)^2
\ge
\frac{T^{-\gamma}}{2\alpha+1}(t-s)^{2\alpha+1},
\qquad 0\le s<t\le T.
\end{equation}
Combining this with statement~$1)(i)$, we conclude that $U^{(2)}$
is a local
$(\alpha+\frac12,\alpha+\frac{1-\gamma}{2},0,T)$-generalized quasihelix.

\bigskip
It remains to prove statements~$2)$ and $3)$. Note that
\eqref{eq:J2-lower}, and hence \eqref{eq:lower-bound}, is valid for all
$\alpha>-1/2$. Therefore, it remains only to establish the corresponding
upper bounds.

\medskip

$2)$ Let $\alpha\in[1/2+\gamma/2,1)$. Since
$(1+z)^\alpha-z^\alpha\le1$, \eqref{Jiji11} gives
\[
J^1\le C(T,\alpha,\gamma)(t-s)^{2\alpha}.
\]
Moreover, by \eqref{int1.12},
\[
J^2\le C(\alpha,\gamma)(t-s)^{2\alpha+1-\gamma}
\le C(T,\alpha,\gamma)(t-s)^{2\alpha}.
\]
Together with \eqref{eq:lower-bound}, this proves the local
$(\alpha+\frac12,\alpha,0,T)$-generalized quasihelix property.

\medskip

$3)$ Let $\alpha\ge1$. Then, by the mean value theorem,
\[
0 < (t-u)^\alpha-(s-u)^\alpha
=
\alpha(\theta_{s,t}-u)^{\alpha-1}(t-s),
\qquad \theta_{s,t}\in(s,t).
\]
Therefore,
\[
J^1
\le
\alpha^2(t-s)^2\int_0^s u^{-\gamma}(t-u)^{2\alpha-2}\,du
\le
\frac{\alpha^2}{1-\gamma}T^{2\alpha-1-\gamma}(t-s)^2.
\]
Also,
\[
J^2
\le
(t-s)^{2\alpha}\frac{t^{1-\gamma}-s^{1-\gamma}}{1-\gamma}
\le
\frac{(t-s)^{2\alpha+1-\gamma}}{1-\gamma}.
\]
Since $2\alpha+1-\gamma>2$, it follows that
\[
\ex\big(U^{(2)}(t)-U^{(2)}(s)\big)^2
=J^1+J^2
\le C(T,\alpha,\gamma)(t-s)^2.
\]
Combining this with \eqref{eq:lower-bound}, we obtain that $U^{(2)}$
is a local $(\alpha+\frac12,1,0,T)$-generalized quasihelix.

The theorem is proved.
\end{proof}

\subsection{Gaussian Volterra processes starting from \texorpdfstring{$-\infty$}{-Inf}}
\label{sec:gvp-inf}

We now consider three specific Gaussian Volterra processes defined by distinct kernels $K(t,s)$ and started from $-\infty$. As before, we analyze their properties in the context of conditions \eqref{loc_quasihelix} and \eqref{loc_gen_quasihelix}.

The first two processes are known as tempered fractional Brownian motion and tempered fractional Brownian motion of the second kind. They were introduced in \cite{Meerschaert2013} and \cite{Sabzikar2018}, respectively, and have been extensively studied in \cite{Azmoodeh2021, Boniece2021, Macioszek2025, MR24, Prykhodko2026, Zhang2024}. These processes are defined via generalized exponential-tempered Riemann--Liouville kernels on the real line for $\alpha>-1/2$ and $\lambda>0$. Since, by construction, the integrands vanish for $u>t$, we obtain the following representations:
\begin{align}\label{gen_tfbm1}
U^{(4)}(t)
&= \int_{-\infty}^{t}
\left[
e^{-\lambda (t-u)}(t-u)^{\alpha}
-
e^{-\lambda (-u)_+}(-u)_+^{\alpha}
\right]\,dW_u,
\qquad t\ge0,
\\
\label{gen_tfbm2}
U^{(5)}(t)
&= \int_{-\infty}^{t}
\biggl[
e^{-\lambda (t-u)}(t-u)^{\alpha}
- e^{-\lambda (-u)_+}(-u)_+^{\alpha}
\\
&\qquad\qquad\qquad\qquad
+ \lambda \int_{0}^{t} e^{-\lambda (v-u)_+}(v-u)_+^{\alpha}\,dv
\biggr]\,dW_u,
\qquad t\ge0.
\notag
\end{align}

The third process was introduced in \cite{WangXiao2022}. For $\gamma\in[0,1)$ and $\alpha>-1/2$, it is given by
\begin{equation}\label{gfbm_full}
U^{(6)}(t)
=
\int_{-\infty}^{t}
|u|^{-\gamma/2}
\left[(t-u)^{\alpha}-(-u)_+^{\alpha}\right]\,dW_u,
\qquad t\ge0.
\end{equation}
It can be decomposed as
\[
U^{(6)}(t)=V(t)+U^{(2)}(t),
\]
where
\[
V(t)
=
\int_{-\infty}^{0}
(-u)^{-\gamma/2}
\left[(t-u)^{\alpha}-(-u)^{\alpha}\right]\,dW_u.
\]
The processes $V$ and $U^{(2)}$ are independent. Since the quasihelix properties of $U^{(2)}$ have already been established, it remains to study the process $V$.

The following theorem collects the quasihelix and generalized quasihelix properties for the processes (\ref{gen_tfbm1})--(\ref{gfbm_full}).

\begin{theorem}\label{theor1-2}
\leavevmode
\begin{enumerate}[1)]
\item
\begin{enumerate}[(i)]
\item
For any fixed $\lambda>0$ and $\alpha\in(-1/2,1/2)$, the processes
$U^{(4)}$ and $U^{(5)}$ are local
$(\alpha+\frac12,0,T)$-quasihelices for any $T>0$.

\item
Let $\alpha=1/2$. Then, for any fixed $\lambda>0$, the processes
$U^{(4)}$ and $U^{(5)}$ are local
$(1,1-\varepsilon,0,T)$-generalized quasihelices for any $T>0$ and any
$\varepsilon\in(0,1)$.

\item
Let $\alpha>1/2$. Then, for any fixed $\lambda>0$, the processes
$U^{(4)}$ and $U^{(5)}$ are local $(1,0,T)$-quasihelices for any $T>0$.
\end{enumerate}

\item
Let $0\le\gamma<1$, $-1/2<\alpha<1/2+\gamma/2$, and $\alpha\ne0$. Then, for any $0<T_1<T_2$, the process $V$ is a local $(1,T_1,T_2)$-quasihelix.
\end{enumerate}
\end{theorem}

\begin{proof}
Statement~$1)(i)$ was proved on the interval $[0,1]$ in
\cite[Theorem~2.7]{Azmoodeh2021}. The extension to an arbitrary interval
$[0,T]$ follows from the corresponding scaling relation for tempered
fractional processes, with constants depending on $T$. The cases
$\alpha\ge1/2$ were established in \cite[Corollary~1]{Prykhodko2026}.
Thus, the quasihelix and generalized quasihelix properties stated in
$1)(i)$--$1)(iii)$ follow.

It remains to prove statement~2). 
In the parameter range $-\frac12+\frac{\gamma}{2}<\alpha<\frac12+\frac{\gamma}{2}$, the assertion follows from \cite[Lemma~2.1]{WangXiao2022}. We provide the following proof in order to cover the full range $-\frac12 < \alpha < \frac12+\frac{\gamma}{2}$, $\alpha\ne0$, and to keep the argument self-contained.

For $0\le s<t$, we have
\[
V(t)-V(s)
=
\int_{-\infty}^{0}(-u)^{-\gamma/2}
\left[(t-u)^\alpha-(s-u)^\alpha\right]\,dW_u.
\]
Hence
\begin{align}
\ex (V(t)-V(s))^2
&=
\int_{-\infty}^{0}(-u)^{-\gamma}
\left[(t-u)^\alpha-(s-u)^\alpha\right]^2\,du
\notag\\
&=
\int_{0}^{\infty}u^{-\gamma}
\left[(t+u)^\alpha-(s+u)^\alpha\right]^2\,du
=:J^4 .
\label{eq:J5}
\end{align}
Putting $h=t-s$ and $r=s/h$, we rewrite $J^4$ as
\begin{align*}
J^4
&=
\int_s^\infty (u-s)^{-\gamma}
\left[(h+u)^\alpha-u^\alpha\right]^2\,du
\\
&=
h^{2\alpha-\gamma+1}
\int_r^\infty (u-r)^{-\gamma}
\left[(1+u)^\alpha-u^\alpha\right]^2\,du .
\end{align*}
By the mean value theorem,
\[
(1+u)^\alpha-u^\alpha
=
\alpha(\theta_u+u)^{\alpha-1},
\qquad \theta_u\in(0,1).
\]
Since $\alpha<1/2+\gamma/2<1$, the exponent $2\alpha-2$ is negative.
Therefore,
\begin{align*}
J^4
&\le
\alpha^2 h^{2\alpha-\gamma+1}
\int_r^\infty (u-r)^{-\gamma}u^{2\alpha-2}\,du
\\
&=
\alpha^2 h^{2}
s^{2\alpha-\gamma-1}
\int_1^\infty (u-1)^{-\gamma}u^{2\alpha-2}\,du .
\end{align*}
The integral is finite precisely under the condition
$\alpha<1/2+\gamma/2$. Since $s\in[T_1,T_2]$, this gives
\[
J^4\le C(T_1,T_2,\alpha,\gamma)(t-s)^2 .
\]

For the lower bound, again using the mean value theorem and the negativity of
$2\alpha-2$, we obtain
\begin{align*}
J^4
&\ge
\alpha^2 h^{2\alpha-\gamma+1}
\int_r^\infty (u-r)^{-\gamma}(1+u)^{2\alpha-2}\,du
\\
&=
\alpha^2 h^{2}s^{2\alpha-\gamma-1}
\int_0^\infty v^{-\gamma}
\left(\frac{t}{s}+v\right)^{2\alpha-2}\,dv .
\end{align*}
Since $T_1\le s<t\le T_2$, we have $t/s\le T_2/T_1$. As
$2\alpha-2<0$,
\[
\left(\frac{t}{s}+v\right)^{2\alpha-2}
\ge
\left(\frac{T_2}{T_1}+v\right)^{2\alpha-2}.
\]
Therefore,
\[
J^4
\ge
c(T_1,T_2,\alpha,\gamma)(t-s)^2 .
\]
Combining the two estimates yields
\[
c(T_1,T_2,\alpha,\gamma)(t-s)^2
\le
\ex (V(t)-V(s))^2
\le
C(T_1,T_2,\alpha,\gamma)(t-s)^2,
\]
which proves that $V$ is a local $(1,T_1,T_2)$-quasihelix.
\end{proof}

\begin{corollary}
Let $0<T_1<T_2$ and $\gamma\in[0,1)$.
Comparing the properties of the process $U^{(2)}$ from Theorem \ref{theor1-1b} and of  process $V$ from Theorem \ref{theor1-2},  we can state the respective properties of the process $U^{(6)}$.
\begin{enumerate}[$(i)$]
\item
If $\alpha\in(-1/2,1/2)$, then $U^{(6)}$ is a local
$(\alpha+\frac12,T_1,T_2)$-quasihelix.

\item
If $\alpha=1/2$, then $U^{(6)}$ is a local
$(1,1-\varepsilon,T_1,T_2)$-generalized quasihelix for any
$\varepsilon\in(0,1)$.

\item
If $\alpha\in(1/2,1/2+\gamma/2)$, then $U^{(6)}$ is a local
$(1,T_1,T_2)$-quasihelix.
\end{enumerate}
\end{corollary}

\appendix
\section{Proof of Theorem \texorpdfstring{\ref{theor1-1b}\,$(i)$}{3.3(i)}}
\label{app}
Denote, as before, $r=\frac{s}{t-s}$.
According to the representations \eqref{Jiji11} and \eqref{int1.12},
\begin{equation}\label{eq:J1+J2}
J^1+J^2 = (t-s)^{2\alpha-\gamma+1}
\bigl(g_1(r) + g_2(r)\bigr),
\end{equation}
where
\[
g_1(r):=\int_{0}^{r} (r-z)^{-\gamma}
\left[(1+z)^{\alpha}-z^{\alpha}\right]^2 dz,
\qquad
g_2(r):=\int_0^1 (z+r)^{-\gamma}(1-z)^{2\alpha}\,dz.
\]

We now consider the asymptotic behavior of the functions $g_1(r)$ and $g_2(r)$ as $r\to \infty$.
First,
\begin{equation}\label{eq:g2conv}
r^\gamma g_2(r)
=
\int_0^1 \left(1+\frac{z}{r}\right)^{-\gamma}(1-z)^{2\alpha}\,dz
\to \int_0^1 (1-z)^{2\alpha}\,dz
=
\frac{1}{2\alpha+1},
\quad\text{as } r\to\infty.
\end{equation}
Next,
\[
r^\gamma g_1(r)
=
\int_0^r \left(1-\frac{z}{r}\right)^{-\gamma}
\left[(1+z)^{\alpha}-z^{\alpha}\right]^2 dz .
\]
We split this integral into the intervals $(0,r/2)$ and $(r/2,r)$. If $0 < z < r/2$, then
$(1 - z/r)^{-\gamma}\le 2^\gamma$.
Since
\[
\int_0^\infty \left[(1+z)^\alpha-z^\alpha\right]^2 dz<\infty,
\qquad -\tfrac12<\alpha<\tfrac12,
\]
the dominated convergence theorem yields
\[
\int_0^{r/2}\left(1-\frac{z}{r}\right)^{-\gamma}
\left[(1+z)^{\alpha}-z^{\alpha}\right]^2 dz
\to
\int_0^\infty \left[(1+z)^\alpha-z^\alpha\right]^2 dz
\quad\text{as } r\to\infty.
\]

It remains to show that the contribution from $(r/2,r)$ is negligible. By the mean value theorem, uniformly for $z\in[r/2,r]$,
\[
\bigl|(1+z)^\alpha-z^\alpha\bigr|\le C z^{\alpha-1}.
\]
Therefore,
\begin{multline*}
\int_{r/2}^{r}\left(1-\frac{z}{r}\right)^{-\gamma}
\bigl[(1+z)^\alpha-z^\alpha\bigr]^2\,dz
\le C\int_{r/2}^{r}\left(1-\frac{z}{r}\right)^{-\gamma}z^{2\alpha-2}\,dz
\\
= C r^{2\alpha-1}
\int_{1/2}^{1}(1-y)^{-\gamma}y^{2\alpha-2}\,dy
\to0, \quad\text{as } r\to\infty,
\end{multline*}
because $2\alpha-1<0$ and $\gamma<1$.
Therefore,
\begin{equation}\label{eq:g1conv}
r^\gamma g_1(r) \to
\int_0^\infty \left[(1+z)^\alpha-z^\alpha\right]^2 dz,
\quad\text{as } r\to\infty.
\end{equation}

Combining the limits \eqref{eq:g2conv} and \eqref{eq:g1conv}, we obtain
\[
g_1(r) + g_2(r) \sim C(\alpha) r^{-\gamma},
\quad\text{as } r\to\infty,
\]
where $C(\alpha)$ is given by \eqref{eq:Calpha}.
Thus, by \eqref{eq:J1+J2},
\[
\ex\bigl(U^{(2)}(t)-U^{(2)}(s)\bigr)^2
\sim C(\alpha) (t-s)^{2\alpha-\gamma+1} r^{-\gamma}
= C(\alpha) s^{-\gamma}(t-s)^{2\alpha+1},
\quad t\downarrow s.
\]
The second representation of $C(\alpha)$ in \eqref{eq:Calpha} is due to \cite[Appendix A]{Mishura}.

Since $s^{-\gamma}$ is bounded above and below by positive constants on
$[T_1,T_2]$, the asymptotic relation \eqref{eq:asymp} implies that there exist
$\delta>0$ and constants $0<c_1<C_1<\infty$ such that
\[
c_1(t-s)^{2\alpha+1}
\le
\ex\bigl(U^{(2)}(t)-U^{(2)}(s)\bigr)^2
\le
C_1(t-s)^{2\alpha+1}
\]
whenever $T_1\le s<t\le T_2$ and $t-s<\delta$. On the compact set
\[
D_\delta=\{(s,t): T_1\le s<t\le T_2,\ t-s\ge\delta\},
\]
the function
\[
F(s,t)=
\frac{\ex\left(U^{(2)}(t)-U^{(2)}(s)\right)^2}{(t-s)^{2\alpha+1}}
\]
is continuous and strictly positive. Hence, it attains a positive minimum
and a finite maximum on $D_\delta$. Consequently, a two-sided bound of the
same form also holds for $t-s\ge\delta$. Combining the two regions and using
the symmetry of the incremental variance, we obtain the local
$(\alpha+\frac12,T_1,T_2)$-quasihelix property.
This proves statement~$(i)$.

\providecommand{\bysame}{\leavevmode\hbox to3em{\hrulefill}\thinspace}
\providecommand{\MR}{\relax\ifhmode\unskip\space\fi MR }
\providecommand{\MRhref}[2]{%
  \href{http://www.ams.org/mathscinet-getitem?mr=#1}{#2}
}
\providecommand{\href}[2]{#2}

\end{document}